\documentclass[12pt]{amsart}

\title{{\bfseries Torsion of rational elliptic curves over different types of cubic fields}}
\author{Daeyeol Jeon and Andreas Schweizer}
\date{}

\address{Daeyeol Jeon, 
Department of Mathematics Education, 
Kongju National University, 
Gongju, Chungnam 314-701, South Korea}
\email{dyjeon@kongju.ac.kr}

\address{Andreas Schweizer,
Department of Mathematics, 
Korea Advanced Institute of Science and Technology (KAIST), 
Daejeon 305-701, South Korea}
\email{schweizer@kaist.ac.kr}

\usepackage{amssymb}
\usepackage{amsmath}

\newcommand{\comment}[1]{} %For comments\usepackage{times}
\DeclareSymbolFont{text}{OT1}{\rmdefault}{m}{n}
\DeclareMathSymbol{\amp}{\mathbin}{text}{"26}

\newtheorem{thm}{Theorem}[section]
\newtheorem{lemma}[thm]{Lemma}

\theoremstyle{definition}

\newcommand{\Q}{\mathbb Q}
\newcommand{\Z}{\mathbb Z}

\begin{document}
{\thanks{\scriptsize This work was supported by the Basic Science Research Program through the National Research Foundation of Korea (NRF) funded by the Ministry of Education(NRF-2016R1D1A1B03934504).}}

\maketitle

\noindent
\begin{abstract}
Let $E$ be an elliptic curve defined over $\Q$, and let $G$ be the torsion group $E(K)_{tors}$ for some cubic field $K$ which does not occur over $\Q$.
In this paper, we determine over which types of cubic number fields (cyclic cubic, non-Galois totally real cubic, complex cubic or pure cubic) $G$ can occur, and if so, whether it can occur infinitely often or not.
Moreover, if it occurs, we provide elliptic curves $E/\Q$ together with cubic fields $K$ so that $G= E(K)_{tors}$.

\end{abstract}


\section{Introduction}\label{sec:Introduction}
A celebrated theorem, finally proved by Mazur \cite{M}, states that the torsion group $E(\Q)_{tors}$ of an elliptic curve $E$ over the rational numbers must be isomorphic to one of the following 15 types:
\begin{equation}\label{eq:rt}
\begin{array}{ll}
\Z/n\Z,&n=1,2,3,\dots,10,12\\
\Z/2\Z\times\Z/2n\Z,&n=1,2,3,4
\end{array}
\end{equation}

Let $E$ be an elliptic curve over $\Q$, and $K$ be a cubic number field.
Najman \cite{N} determined that $E(K)_{tors}$ is one of the following 20 types:
\begin{equation}\label{eq:ct}
\begin{array}{ll}
\Z/n\Z,&n=1,2,3,\dots,10,12,13,14,18,21\\
\Z/2\Z\times\Z/2n\Z,&n=1,2,3,4,7
\end{array}
\end{equation}
Moreover, he showed that the elliptic curve 162B1 over $\Q(\zeta_9)^+$ is the unique rational elliptic curve
with torsion $\Z/21\Z$ over a cubic field, and for all the other groups $G$ in the list \eqref{eq:ct}, there exist infinitely many rational elliptic curves that have torsion $G$ over some cubic field.

%Note that $G$ listed in \eqref{eq:ct} but not in \eqref{eq:rt} {\it non-trivial}. Explicitly, such $G$ is equal to $\Z/n\Z$ with $n=13,14,18,21$ or $\Z/2\Z\times\Z/14\Z$.

Later, Gonz\'alez-Jim\'enez, Najman and Tornero \cite{GNT} determined the set of possible torsion structures over a cubic field of a rational elliptic curve such that $E(\Q)_{tors}= G$ for each $G$ listed in \eqref{eq:rt}. Also, they studied the number of cubic fields $K$ such that $E(\Q)_{tors}\neq E(K)_{tors}$.

Recently, Derickx and Najman \cite{DN} determined all the possible torsion groups of elliptic curves over cyclic cubic fields, over non-cyclic totally real cubic fields and over complex cubic fields.
Also Gonz\'alez-Jim\'enez \cite{G} gave an explicit description of the possible torsion growth of rational elliptic curves with complex multiplication over cubic fields.

%Also Gonz\'alez-Jim\'enez \cite{G} gave a description of the possible torsion growth over cubic fields and a completely explicit description of this growth in terms of some invariants attached to a given rational elliptic curve with complex multiplication.

Let $K$ be a cubic number field.
Then $K=\Q(\alpha)$ for some $\alpha$ whose minimal polynomial is a cubic polynomial $f(x)$.
If all three roots of $f(x)$ are real, $K$ is called a {\it totally real cubic field}, and if $f(x)$ has a non-real root, it is called {\it complex cubic field}.
Moreover, if $K$ contains all three roots of $f(x)$, i.e., $K$ is a Galois extension of $\Q$, then $K$ is called a {\it cyclic cubic field}. Indeed, a cyclic cubic field must be a totally real cubic field.
Finally, if $K$ can be obtained by adjoining the cube root $\sqrt[3]{n}$ of a positive integer $n$, then $K$ is called a {\it pure cubic field}.

In this paper, we determine whether a torsion group $G$ not occurring over $\Q$ can occur over which types of cubic fields, and if so, whether it can occur infinitely often or not.
Moreover, if it occurs, we provide elliptic curves $E$ together with cubic fields $K$ so that $G= E(K)_{tors}$.

Finally, we note that $K$-{\it rational} means defined over a field $K$, and {\it rational} without $K$ means $\Q$-rational.

\section{Results}

In this section, for each torsion group $G$ not occurring over $\Q$, we give results on it case by case.
However we don't need to treat the case $E(K)_{tors}=\Z/21\Z$ as stated in the Introduction.

Similar to Cremona and Watkins \cite{CW}, for a positive integer $N$, we will describe models for elliptic curves with $N$-isogenies as 
$$y^2 = x^3 + A_N(t,U)x + B_N(t,U)$$
where $A_N(t,U)=f(t)U^2$ and $B_N(t,U)=g(t)U^3$, $t$ and $U$ are parameters and $f$ and $g$ are functions of $t$.
A fixed value of $t$ corresponds to a family of quadratic twists of an elliptic curve (all with the same $j$-invariant), a fixed value of $t$ and a value of $U$ up to squares defines an elliptic curve, 
and a $(t,U)$-pair defines a model.

\begin{lemma} \label{lem:GNT} {}\
\begin{itemize} 	
\item[(a)] \cite[Lemma 2.5]{GNT}\label{lem:GNT} Let $p$ be prime, $f$ a $p$-isogeny on $E$ over $\Q$, and let ${\rm ker}(f)$ be generated by $P$.
Then the field of definition $\Q(P)$ of $P$ (and all of its multiples) is a cyclic (Galois) extension of $\Q$ of order dividing $p-1$.
\item[(b)] \cite[Lemma 18]{N} Let $E$ be an elliptic curve over $\Q$ and $K$ a cubic number field. If $P$ is a $K$-rational $n$-torsion point
of $E$ where $n$ is odd and not divisible by $3$, then $P$ generates a $\Q$-rational $n$-isogeny of $E$.
\end{itemize} 
\end{lemma}

\begin{center}
2.1.\,$E(K)_{tors}=\Z/13\Z$
\end{center}

%Suppose $E$ is a rational elliptic curve with $E(K)_{tors}=\Z/13\Z$ for some cubic field $K$. 
%Then $E(\Q)_{tors}=\{O\}$ is trivial. 
%By Lemma \ref{lem:GNT}, $K$ should be a cyclic cubic field.
%We give a somewhat more direct proof of this fact as follows:
 
\begin{lemma}\label{lem:13} 
Suppose $E$ is a rational elliptic curve with $E(K)_{tors}=\Z/13\Z$ for some cubic field $K$.
Then $K$ is a cyclic cubic field.
\end{lemma}
\begin{proof}
By Mazur's Theorem, $E(\Q)_{tors}=\{O\}$ is trivial. So Lemma \ref{lem:GNT} shows that $K$ must be cyclic.
\end{proof}
 
%\begin{proof}
%Assume that $K$ is not Galois.
%Let $L$ be the Galois closure of $K$ over $\Q$, so ${\rm Gal}(L/\Q)=S_3$.
%Then the $13$-torsion of $E$ over $L$, namely $E[13](L)$, is still $1$-dimensional, for otherwise by the Weil pairing
%$L$ should contain a primitive 13-th root of unity $\zeta_{13}$, which is a contradiction.
%Now ${\rm Gal}(L/\Q)$ acts on ${\mathcal C}=E[13](L)$, and the image of ${\rm Gal}(L/\Q)$ is in ${\rm Aut}(\mathcal C)$ which is a cyclic group of order 12.
%Thus the image of ${\rm Gal}(L/\Q)$ must be either trivial or of order 2.
%If it is trivial, all the points of $\mathcal C$ are rational, which is a contradiction to $E(\Q)_{tors}=\{O\}$.
%If it is a cyclic group of order 2, it corresponds to $P\mapsto -P$.
%Hence the unordered pair $\{P,-P\}$ is defined over $\Q$, and it gives a non-cuspidal rational point on $X_1(13)$, which is a contradiction to Mazur's theorem.
%\end{proof}

Now we will show that $\Z/13\Z$ occurs infinitely often over cyclic cubic fields and construct an infinite family of elliptic curves whose torsion group is $\Z/13\Z$ over cyclic cubic fields.
By the computation in \cite{CW}, we have a family of elliptic curves $E_{t,U}$ which have $13$-isogenies over cyclic cubic fields $K_{t,U}$ as follows:

\begin{equation}\label{eq:iso13}
E_{t,U}:\ \ y^2 = x^3 + A_{13}(t,U)x + B_{13}(t,U),
\end{equation}
where 
\begin{align*}
A_{13}(t,U)=&-27(t^4-t^3+5t^2+t+1)(t^2+1)^2(t^8-5t^7+7t^6-5t^5+5t^3+7t^2\\
&+5t+1)U^2,\\
B_{13}(t,U)=&-54(t^4-t^3+5t^2+t+1)(t^2+1)^4(t^{12}-8t^{11}+25t^{10}-44t^9+40t^8\\
&+18t^7-40t^6-18t^5+40t^4+44t^3+25t^2+8t+1)U^3
\end{align*}
with $t\neq 0$.

$K_{t,U}=\Q(\alpha_{t,U})$ where $\alpha_{t,U}$ is a root of an irreducible cubic polynomial $a_3(t,U)x^3+a_2(t,U)x^2 +a_1(t,U)x+a_0(t,U)$ for some rational numbers $t$ and $U$ where
\begin{align*}
a_3(t,U)=&t^{12},\\
a_2(t,U)=&9t^8(t-1)^2(t^2+1)(t^4-t^3+5t^2+t+1)U,\\
a_1(t,U)=&27t^4(t^2+1)^2(t^4-t^3+5t^2+t+1)(t^8-5t^7+15t^6-29t^5+16t^4-3t^3\\
&-9t^2-3t+1)U^2,\\
a_0(t,U)=&27(t^2+1)^3(t^4-t^3+5 t^2+t+1)(t^{14}-8 t^{13}+38 t^{12}-124 t^{11}+245 t^{10}\\
&-326 t^9+228 t^8+120 t^7+12 t^6+38 t^5-43 t^4-80 t^3-34 t^2-4 t+1)U^3.
\end{align*}

For simplicity, let $E_t$, $K_t$ and $\alpha_t$ denote $E_{t,1}$, $K_{t,1}$ and $\alpha_{t,1}$, respectively.
Now we will find a quadratic twist $E_{t,U}$ of $E_t$ that has a $K_t$-rational 13-torsion point.
Note that $\alpha_t$ is the $x$-coordinate of a $13$-torsion point on $E_t$.
Let $\beta_t$ denote its $y$-coordinate.
By a quadratic twist $U=d^2$, $E_t$ becomes $E_{t,U}$ and $(\alpha_t,\,\beta_t)$ maps to $(d^2\alpha_t,\, d^3\beta_t)$.
Thus $U=d^2\in\Q$ and $d^3\beta_t\in K_t$, hence $d\beta_t\in K_t$.
One can easily check that $E_{t,U}$ has a 13-torsion point over $K_t$ if and only if $U=d^2\in\Q$ and $d\beta_t\in K_t$.
Now $\beta_t^2=\alpha_t^3+A_{13}(t,1)\alpha_t+B_{13}(t,1)=a_2(t)\alpha_t^2+a_1(t)\alpha_t+a_0(t)\in K_t$ where
\begin{align*}
a_2(t)=&-9(t^2+1)(t^4-t^3+5t^2+t+1)(t^5-t^4)^2/t^{12},\\
a_1(t)=&-18(t^2+1)(t^4-t^3+5t^2+t+1)(3t^9-12t^8+24t^7-42t^6+15t^5-33t^4\\
&-12t^3-6t^2-6t-3)(t^5-t^4)/t^{12},\\
a_0(t)=&-9(t^2+1)(t^4-t^3+5t^2+t+1)(3t^9-12t^8+24t^7-42t^6+15t^5-33t^4\\
&-12t^3-6t^2-6t-3)^2/t^{12}.
\end{align*}
Hence we have
\begin{equation}\label{eq:beta1}
d^2\beta_t^2=d^2a_2(t)\alpha_t^2+d^2a_1(t)\alpha_t+d^2a_0(t).
\end{equation}

On the other hand, $d^2\beta_t^2=(d\beta_t)^2$ is a square in $K_t$, hence we have
\begin{equation}\label{eq:beta2}
d^2\beta_t^2=(b_2(t)\alpha_t^2+b_1(t)\alpha_t+b_0(t))^2=c_2(t)\alpha_t^2+c_1(t)\alpha_t+c_0(t)
\end{equation}
for some $b_i(t),c_i(t)\in \Q(t)$ with $i=0,1,2$.
By comparing \eqref{eq:beta1} and \eqref{eq:beta2} and using {\sc Maple}, we can obtain the following:
\begin{align*}
U =&\frac{-1}{(t^2+1)(t^4-t^3+5t^2+t+1)},\\
b_2(t) =& 0,\\
b_1(t) =& \frac{3(t-1)}{t^2},\\
b_0(t)=&\frac{9(t^2+1)(t^7-4t^6+7t^5-10t^4-2t^3-t^2-2t-1)}{t^6}.
\end{align*}
Finally, by letting $t=u$ and $U =\frac{-1}{(u^2+1)(u^4-u^3+5u^2+u+1)}$ in \eqref{eq:iso13}, we have an infinite family of rational elliptic curves $E_u$ over cyclic cubic fields $K_u$ with $E_u(K_u)_{tors}=\Z/13\Z$ as follows:

$\bullet\, E_u:\ \ y^2=x^3+A(u)x+B(u)$ where
\begin{align*}
A(u)=&-27(u^8-5 u^7+7 u^6-5 u^5+5 u^3+7 u^2+5 u+1)/(u^4-u^3+5u^2+u+1),\\
B(u)=&54(u^2+1)(u^{12}-8 u^{11}+25 u^{10}-44 u^9+40 u^8+18 u^7-40 u^6-18 u^5+40 u^4\\
&+44 u^3+25 u^2+8 u+1)/(u^4-u^3+5u^2+u+1)^2
\end{align*}
with $u(u^4 -u^3 +5u^2 +u+1)\neq 0$.

$\bullet\, K_u=\Q(\alpha_u)$ where $\alpha_u$ is a root of an irreducible polynomial $a_3(u)x^3+a_2(u)x^2 +a_1(u)x+a_0(u)$ for some rational number $u$ where
\begin{align*}
a_3(u)=&u^{12}(u^4-u^3+5u^2+u+1)^2,\\
a_2(u)=&-9u^8(u-1)^2(u^4-u^3+5u^2+u+1)^2,\\
a_1(u)=&27u^4(u^4-u^3+5u^2+u+1)(u^8-5u^7+15u^6-29u^5+16u^4-3u^3-9u^2\\
&-3u+1),\\
a_0(u)=&-27 u^{14}+216 u^{13}-1026 u^{12}+3348 u^{11}-6615 u^{10}+8802 u^9-6156 u^8\\
&-3240 u^7-324 u^6-1026 u^5+1161 u^4+2160 u^3+918 u^2+108 u-27.\\
\end{align*}

By this result together with Lemma \ref{lem:13}, we have the following result:

\begin{thm} Suppose $E$ is a rational elliptic curve with $E(K)_{tors}=\Z/13\Z$ for some cubic field $K$.
Then $K$ is a cyclic cubic field.
Moreover, there exist infinitely many non-isomorphic rational elliptic curves $E$ and cyclic cubic fields $K$ so that $E(K)_{tors}=\Z/13\Z$.
\end{thm}

\begin{center}
2.2.\, $E(K)_{tors}=\Z/14\Z$
\end{center}

In this subsection, we are interested in the case where $\Z/14\Z$ is the full torsion of $E(K)$.
Elliptic curves with $E(K)_{tors}=\Z/2\Z\times\Z/14\Z$ will be treated in Subsection 2.4.
Suppose $E$ is a rational elliptic curve with $E(K)_{tors}=\Z/14\Z$ for some cubic field $K$.
By \cite[Theorem 1.2]{GNT}, $E(\Q)_{tors}=\Z/2\Z$ or $E(\Q)_{tors}=\Z/7\Z$.

First consider the case $E(\Q)_{tors}=\Z/2\Z$. Let $P$ be a $7$-torsion point in $E(K)_{tors}$.
Then $K=\Q(P)$, hence $K$ is a cyclic cubic field by Lemma \ref{lem:GNT}.
Also $E(K)_{tors}$ defines a rational $14$-isogeny on $E$, hence it gives rise to a non-cuspidal rational point on $X_0(14)$.
By \cite{K}, $X_0(14)$ contains only two such points.
Now we explain how to find the rational elliptic curves corresponding to them.
By the method explained in \cite{JKL2}, we can construct a map from $X_1(14)\to X_0(14)$.
Note that $X_1(14)$ and $X_0(14)$ are defined by
\begin{align}\label{eq:14-1}
&X_1(14):\,\, y^2 + (x^2 + x)y + x=0,\\ \label{eq:14-2}
&X_0(14):\,\, v^2+(u+3)v+u^3+6u+8=0.
\end{align}
Then the natural map $\phi:X_1(14)\to X_0(14)$ is defined by
\begin{equation}\label{eq:map14}
(u,v)=\phi(x,y)=\left(\frac{-1-y+y^3}{y^2}, \frac{-1-x^2-x^3-y-y^3-3xy-xy^2}{xy}\right).
\end{equation}
We can find 6 rational points satisfying \eqref{eq:14-2} which correspond to 2 non-cuspidal points and 4 cusps.
The rational points corresponding to non-cuspidal points are $P_1:=(-2,3)$ and $P_2:=(-9,-25)$.
By using \eqref{eq:map14}, we have that the points lying above $P_1$ and $P_2$ which satisfy \eqref{eq:14-1} are $Q_1:=(1-\alpha-\alpha^2,\,\alpha)$ and $Q_2:=\left(-3 + 2 \alpha +  2 \alpha^2,\,1-2\alpha^2\right)$, respectively, where $\alpha$ is a root of the irreducible cubic polynomial $x^3+2x^2-x-1$.
Actually,
$$K=\Q(\alpha)=\Q(\zeta_7)^+,$$
the maximal real subfield of the $7$-th cyclotomic field.

By using the rational maps in Table 7 and p.~1133 of \cite{S}, we obtain the elliptic curves  $E_1$ and $E_2$ corresponding to $Q_1$ and $Q_2$, respectively, as follows:
\begin{align*}
E_1:\,& y^2+\left(\frac{5}{7}\alpha^2+\frac{2}{7}\alpha +\frac{3}{7}\right)xy+\left(\alpha^2-\frac{1}{7}\alpha-\frac{3}{7}\right)y=x^3+\left(\alpha^2-\frac{1}{7}\alpha-\frac{3}{7}\right)x^2\\
E_2:\,& y^2-\left(\frac{13}{7}\alpha^2+\frac{22}{7}\alpha -\frac{23}{7}\right)xy-\left(\frac{4}{7}\alpha^2+\frac{12}{7}\alpha+1\right)y=x^3-\left(\frac{4}{7}\alpha^2+\frac{12}{7}\alpha+1\right)x^2
\end{align*}
They are $K$-rational elliptic curves with torsion $\Z/14\Z$ over $K$.
Also the $j$-invariants of $E_1$ and $E_2$ are $-3375$ and $16581375$, respectively.
By using LMFDB\cite{L}, we can find that 
$$ 49A3:\ \ \ y^2 +xy = x^3 -x^2 -107x+552$$
and 
$$ 49A4:\ \ \ y^2 +xy = x^3 -x^2 -1822x+30393$$

are elliptic curves with $j$-invariants $-3375$ and $16581375$, respectively.
By the computer algebra system {\sc Maple}, we confirm that $E_1$ and $E_2$ are isomorphic to 49A3 and 49A4 over $K$, respectively.
Thus we can conclude that 49A3 and 49A4 are two rational elliptic curves corresponding to the two non-cuspidal rational points on $X_0(14)$.

However, proving that they are up to $\Q$-isomorphisms the only rational elliptic curves with torsion $\Z/14\Z$ over $\Q(\zeta_7)^+$ 
requires some more justification. A priori, the modular curve only tells us that they are the only such curves up to $\overline{\Q}$-isomorphisms.
We still have to exclude the possibility that a quadratic twist, i.e. a rational elliptic curve isomorphic to 49A3 or 49A4 only over
$\Q(\sqrt{d})$, might also have torsion $\Z/14\Z$ over $K$. 

The quadratic twist multiplies the $y$-coordinate of the $7$-torsion point
with a quadratic irrationality, and hence moves it out of $K$. But we have to take into account the possibility that at the same time another 
$7$-torsion point might become $K$-rational. This would imply that over the composite field $\Q(\sqrt{d})K$ the curve 49A3 resp. 49A4 has
two independent $7$-torsion points; so by the Weil pairing $\Q(\sqrt{d})K=\Q(\zeta_7)$, and hence $d=-7$.
Now we can invoke \cite[Theorem 1.1]{GL}, which says (among other things) that no rational elliptic curve can
acquire its full $7$-torsion over $\Q(\zeta_7)$. Alternatively, one can use {\sc Magma} to check directly that
over $K=\Q(\zeta_7)^+$ the $\Q(\sqrt{-7})$-twists of 49A3 and 49A4 still only have torsion $\Z/2\Z$.

We also mention that 49A3 and 49A4 are $CM$-curves with complex multiplication by $\Z[\frac{1+\sqrt{-7}}{2}]$ resp. $\Z[\sqrt{-7}]$.

Next consider the case $E(\Q)=\Z/7\Z$.
Suppose $E$ is defined by a short Weierstrass form.
In this case, if we adjoin the $x$-coordinate $x(P)$ of a point $P$ of order 2, we have a cubic field $K=\Q(x(P))$ over which $E(K)=\Z/14\Z$.
Note that by \cite[Theorem 1.2]{GNT} (or rather by the proof of \cite[Proposition 29]{N}), $E(\Q)_{tors}\cong\Z/7\Z$ can never give
$E(K)_{tors}\cong\Z/2\Z\times\Z/14\Z$, so $K$ is automatically non-Galois.

Since $X_1(7)$ is a rational curve, it contains infinitely many rational points, hence there exists an infinite family of elliptic curves with $7$-torsion.
One can find the parametrization of such curves in \cite[Table 3]{Ku} as follows;

$\bullet\ \  E_u:\,y^2-(u^2-u-1)xy-(u^3-u^2)y=x^3-(u^3-u^2)x^2$

\noindent with discriminant 
$$\Delta_u= u^7(u-1)^7(u^3-8u^2+5u+1)\neq 0.$$ 

We point out that in \cite{N} the constant term in the cubic factor of this discriminant carries an incorrect sign and that this slip propagates through 
that paper (\cite[p.262]{N} and \cite[p.265]{N}). However, it seems that this is merely a typo and that for the computer calculations the correct formula 
has been used. For example on page 265 in the proof that $E(\Q)_{tors}\cong\Z/7\Z$ never produces $E(K)_{tors}\cong\Z/2\Z\times\Z/14\Z$ the printed 
(i.e. incorrect) formula would lead to a Jacobian of rank $1$. On page 262 (in the proof that $28$-torsion cannot occur) the mistake does not affect the 
outcome of the computation.

Note that $E_u$ is isomorphic to the elliptic curve defined by
%$$y^2=f_u(x):=x^3 + \frac{1}{4}(u^4 - 6 u^3+ 3 u^2+2 u+1) x^2 + \frac{1}{2}u^2(u-1) (u^2-u-1)x+\frac{1}{4}u^4(u-1)^2.$$
\begin{align*}
y^2=f_u(x):=&x^3-\frac{1}{3}(u^2-u+1)(u^6-11u^5+30u^4-15u^3-10u^2+5u+1)x\\
&+\frac{2}{27}(u^{12}-18 u^{11}+117 u^{10}-354 u^9+570 u^8-486 u^7+273 u^6\\
&-222 u^5+174 u^4-46 u^3-15 u^2+6 u+1)
\end{align*}
Let $\alpha_u$ be a root of an irreducible polynomial $f_u(x)$ for some rational number $u$.
Then $K_u:=\Q(\alpha_u)$ is a cubic field, so that $E_u(K_u)_{tors}=\Z/14\Z$.
Let $r_1,r_2,r_3$ be the three real roots of $u^3-8u^2+5u+1=0$ with $r_1<r_2<r_3$, then $r_1<0<r_2<1<r_3$.
Put $I:=(-\infty,\,r_1)\cup (0,\,r_2) \cup (1,\, r_3)$, $J:=(r_1,\,0)\cup (r_2,\,1) \cup (r_3,\,\infty)$.
Then $\Delta_u<0$ for $u \in I$ and $\Delta_u>0$ for $u \in J$, hence $E_u$ has torsion $\Z/14\Z$ over complex cubic field $K_u$ when $u\in I\cap\Q$ and over totally real, but non-Galois cubic fields $K_u$ when $u\in J\cap\Q$.

Finally, we consider the torsion of $E_u$ over pure cubic fields.
For that we need the following easy fact.
\begin{lemma}\label{lem:pure} The discriminant of a pure cubic number field is of the form $-27d^2$ for some $d\in\Q$.
\end{lemma}
\begin{proof} Let $K=\Q(\sqrt[3]{n})$ and $f(x)=x^3-n$ for some positive integer $n$.
Then the discriminant of $f$ is given by $D(f)=-R(f,f')$ where $R(f,f')$ is the resultant of $f$ and $f'$.
Note that
$$R(f,f')=\left|\begin{array}{rrrrr}1&0&0&-n&0\\0&1&0&0&-n\\ 3&0&0&0&0 
\\0&3&0&0&0 \\ 0&0&3&0&0\end{array} \right|=27n^2.$$
Thus the result follows.
\end{proof}

Suppose $u\in I\cap\Q$. By Lemma \ref{lem:pure} a necessary condition for $K_u$ to be a pure cubic field is
\begin{equation}\label{eq:de}
-27k^2=\Delta_u=u^7(u-1)^7(u^3-8u^2+5u+1)
\end{equation}
for some $k\in\Q$.
By letting $v=\frac{9k}{u^3(u-1)^3}$ in \eqref{eq:de}, we have the following equation:
$$v^2=-3u(u-1)(u^3-8u^2+5u+1),$$
which defines a hyperelliptic curve $C$ of genus 2.
Note that the Jacobian $J(C)$ is of rank 0.
Applying the Chabauty method implemented in the computer algebra system {\sc Magma}, we obtain that all the rational points are $(0,\,0),\,(1,\,0),\infty$.
However, these points cannot give pure cubic fields.
Therefore, the torsion $\Z/14\Z$ cannot occur over pure cubic fields.

\begin{thm} Suppose $E$ is a rational elliptic curve with $E(K)_{tors}=\Z/14\Z$ over some cubic field $K$.
\begin{itemize}
\item[(a)] If $E(\Q)_{tors}=\Z/2\Z$, then $E$ is one of {\rm 49A3} and {\rm 49A4}, and $K=\Q(\zeta_7)^+$.
\item[(b)] If $E(\Q)_{tors}=\Z/7\Z$, then there exist infinitely many non-isomorphic rational elliptic curves $E$ over both totally real cubic and complex cubic fields $K$ so that $E(K)_{tors}=\Z/14\Z$.
But there is no rational elliptic curve $E$ so that $E(K)_{tors}=\Z/14\Z$ over a pure cubic field $K$.
\end{itemize}
\end{thm}

\begin{center}
2.3.\, $E(K)_{tors}=\Z/18\Z$
\end{center}

Suppose $E$ is a rational elliptic curve with $E(K)_{tors}=\Z/18\Z$ for some cubic field.
By \cite[Theorem 1.2]{GNT}, $E(\Q)_{tors}=\Z/6\Z$ or $E(\Q)_{tors}=\Z/9\Z$.

First consider the case $E(\Q)_{tors}=\Z/6\Z$. We have the following result:
\begin{lemma}\label{tor18} Suppose $E$ is a rational elliptic curve with $E(K)_{tors}=\Z/18\Z$ for some cubic field $K$. 
If $E(\Q)_{tors}=\Z/6\Z$, then $K$ is a cyclic cubic field.
\end{lemma}
\begin{proof} 
Let $Q$ be a $K$-rational $18$-torsion point of $E$. Then $P=6Q$ must be one of the two $\Q$-rational $3$-torsion points of $E$.
If not, $E$ would have two independent $K$-rational $3$-torsion points which by the Weil pairing would lead to the contradiction
$\zeta_3 \in K$. 
\par 
Over $\overline{\Q}$ there are exactly nine $9$-torsion points $R$ of $E$ with $3R=P$. Always three of them are multiples of each 
other (and hence generate the same field extension of $\Q$) and lie in the same cyclic $9$-isogeny. So, fixing the $K$-rational
$9$-torsion point $R=2Q$, we see that $K/\Q$ is Galois if and only if the Galois conjugates of $R$ are $4R$ and $7R$. If $K/\Q$
is not Galois, then each of $R$ and its two Galois conjugates must lie in a different one of the three cyclic $9$-isogenies
containing $P$; so in this case none of the cyclic $9$-isogenies containing $P$ can be $\Q$-rational.
\par 
Next we note that if $K/\Q$ is not Galois, then $E$ has a $\Q$-rational $3$-isogeny $\langle S\rangle$ different from $\langle P\rangle$.
If not, then by \cite[Proposition 14]{N} $E$ has a $\Q$-rational $9$-isogeny and we can take the $3$-isogeny it contains, which, by what
was just discussed, for non-Galois $K$ is different from $\langle P\rangle$.
\par 
Now assume that $K/\Q$ is not Galois, let $L$ be the Galois closure of $K/\Q$ and $Gal(L/\Q)=\langle \sigma, \tau\rangle\cong S_3$
where $\sigma$ is an automorphism of order $3$ and $\tau$ is the involution fixing $K$. Then
$$ \tau(R)=R,\ \ \tau(S)=-S,\ \ \sigma(S)=S\ \ \ \hbox{\rm and}\ \ \ \sigma(R)=\alpha R+\beta S$$
with $\alpha\in\{1,4,7\}$ and $\beta\in\{1,2\}$. Replacing $S$ by $-S$ if necessary, we can assume $\beta=1$. Then the relation
$(\sigma\tau)^2=id$ forces $\alpha=1$. 
\par 
Now we consider the elliptic curve $\widetilde{E}=E/\langle S\rangle$, that is, the image of $E$ under the $\Q$-rational $3$-isogeny
whose kernel is $\langle S\rangle$. The $Gal(\overline{\Q}/\Q)$-orbit of $R$ consists of the points $R$, $R+S$ and $R+2S$, which all
map to the same point on $\widetilde{E}$. So the image of $R$ on $\widetilde{E}$ is a $\Q$-rational point, and still a $9$-torsion 
point (as $3R=P$ is not in the kernel). But $\widetilde{E}$ also inherits a $\Q$-rational $2$-torsion point. So all in all 
$\widetilde{E}$ is an elliptic curve over $\Q$ with a $\Q$-rational $18$-torsion point. This finally is the desired contradiction.
\end{proof}

Suppose $E$ is a rational elliptic curve with $E(K)_{tors}=\Z/18\Z$ for some cyclic cubic field.
Then ${\rm Gal}(K/\Q)$ acts on $E(K)_{tors}=\Z/18\Z$.
Thus $E(K)_{tors}$ defines a rational cyclic $18$-isogeny on $E$, hence we get a non-cuspidal rational point on $X_0(18)$.
%Note that the natural map $X_1(18)\to X_0(18)$ is of order 3.

Conversely, suppose there is a non-cuspidal rational point on $X_0(18)$.
This corresponds to a rational elliptic curve $E$ with a rational cyclic $18$-isogeny.
Here we assume $E$ is defined by a short Weierstrass equation.
Note that the underlying 2-torsion point is rational, but some elements of ${\rm Gal}(\overline\Q/\Q)$ might map the underlying 3-torsion point $P$ to its inverse.
Thus $x(P)\in\Q$, but the $y$-coordinate $y(P)$ of $P$ might be quadratic over $\Q$.
After a suitable quadratic twist with $y(P)$, we get a new rational elliptic curve $E$ with a rational 3-torsion point.
The 2-torsion point and 18-isogeny are still rational.
Thus $E$ has an 18-isogeny $\phi$ whose kernel ${\rm ker}(\phi)$ contains a rational 6-torsion point.
Let $Q\in{\rm ker}(\phi)$ be a point of order 18 and put $K=\Q(x(Q))$.
If $y(Q)$ is not defined over $K$, then there exists an element $\sigma\in{\rm Gal}(\overline{K}/K)$ which maps $Q$ to $-Q$.
Since $P=6Q$, $\sigma$ maps $P$ to $-P$ which is impossible because $P$ is defined over $\Q$.
Thus  $K$ is actually equal to $\Q(Q)$.
Then the pairs $(E,\,Q)$ and $(E,\,\langle Q\rangle)$ correspond to a $K$-rational point on $X_1(18)$ and a rational point on $X_0(18)$, respectively.
Since the natural map $X_1(18)\to X_0(18)$ is a Galois covering of degree 3 and it maps $(E,\,Q)$ to $(E,\,\langle Q\rangle)$, $K$ should be a cyclic cubic field.
Thus we have a rational elliptic curve $E$ over a cyclic cubic field $K$ with $E(K)_{tors}=\Z/18\Z$.

Since $X_0(18)$ is a curve of genus 0 with a rational point, it has infinitely many rational points.
Thus there exist infinitely many rational elliptic curves $E$ over cyclic cubic fields $K$ so that $E(K)_{tors}=\Z/18\Z$.
For obtaining such an infinite family, we don't compute directly $18$-isogenies because it requires a big computation.
Instead, we use $9$-isogenies.
By the computation in \cite{CW}, we have a family of elliptic curves $E_{t,U}$ with 9-isogenies as follows:

\begin{equation}\label{eq:9}
E_{t,U}:\ \ y^2 =x^3 + A_{9}(t,U)x + B_{9}(t,U),
\end{equation}
where 
\begin{align*}
A_{9}(t,U)&=-2187 (t+1)^3 (9 t^3+27 t^2+27 t+1) U^2,\\ 
B_{9}(t,U)&=-39366 (t+1)^3 (27 t^6+162 t^5+405 t^4+504 t^3+297 t^2+54 t-1) U^3.
\end{align*}

First, we will find a family of elliptic curves $E_s$ with 9-isogenies and underlying rational 3-torsion by choosing appropriate $U$.
Using {\sc Magma}, we obtain that a linear factor of the 3-division polynomial of $E_{t,1}$ is $x + 81t^3 + 243t^2 + 243t + 81$.
By putting $x=-81t^3-243t^2-243t-81$ in \eqref{eq:9} with $U=1$, we have that the square of the $y$-coordinate of a 3-torsion point is equal to $-2^43^9(t+1)^3$.
Now we put $t=s$ and $U=-3(s+1)$ in \eqref{eq:9}, then up to a $\Q(s)$-rational isomorphism $E_{t,U}$ becomes the following:

$ E_s:\ \ y^2 =x^3 +A(s)x + B(s),$
where 
\begin{align}\label{eq:9-2}
A(s)&=-3(s+1)(9s^3+27s^2+27s+1),\\ \nonumber
B(s)&=2(27s^6+162s^5+405s^4+504s^3+297s^2+54s-1),
\end{align}
which has a rational 3-torsion point for any rational number $s\neq -1$.

Using {\sc Magma} we compute that a cubic factor of the 9-division polynomial of $E_s$ is given by
\begin{align*}
F_s(x):=&x^3+(-9s^2 - 30s - 33)x^2 + (27s^4 + 180s^3 + 450s^2 + 516s + 219)x\\ 
   &-27s^6 - 270s^5 - 1053s^4 - 2196s^3 - 2565s^2 - 1566s - 323,\end{align*}
and its discriminant is given by $2^{12}3^{4}(s^2+3s+3)^2$, which is a perfect square.

Thus, for a rational number $s$ such that $F_s(x)$ is irreducible, $E_s$ has a 9-torsion point $P_s$ whose $x(P_s)$ is contained in a cyclic cubic field $K_s:=\Q(\alpha_s)$ where $\alpha_s$ is a root of $F_s(x)$.
Indeed, $P_s$ is defined over $K_s$, for otherwise there exists an element $\sigma\in{\rm Gal}(\overline{K_s}/K_s)$ which maps $P_s$ to $-P_s$, and then the 3-torsion point $3P_s$ maps to $-3P_s$ which is impossible because $3P_s$ is rational.

On the other hand, if $E_s$ has a rational 2-torsion point, then $E_s$ must have a rational 6-torsion point because $E_s$ has a rational 3-torsion point.
In general $f_s(x):=x^3+A(s)x+B(s)$ does not have a linear factor in $x$ over $\Q$.
However, substituting $s=s(u):=\frac{u^3-3u^2}{3u-3}$, $f_{s(u)}(x)$ splits into  a product of a linear factor and a quadratic factor over $\Q$.
Thus $E_{s(u)}$ has a rational 2-torsion point, hence we finally have an infinite family of elliptic curves $E_u$ over cyclic cubic fields $K_u$ so that $E_u(K_u)_{tors}=\Z/18\Z$ and $E_u(\Q)_{tors}=\Z/6\Z$ as follows:

$\bullet\, E_u:\ \ y^2 =x^3+A(u)x + B(u),$
where 
\begin{align*}
A(u)=&-27(u^3-3 u^2+3 u-3) (u^9-9 u^8+36 u^7-90 u^6+162 u^5\\
&-216 u^4+192 u^3-90 u^2+9 u-3),\\
B(u)=&54(u^6-6 u^5+15 u^4-24 u^3+27 u^2-18 u-3)(u^{12}-12 u^{11}\\
&+66 u^{10}-228 u^9+567 u^8-1080 u^7+1596 u^6-1800 u^5\\
&+1503 u^4-900 u^3+378 u^2-108 u+9).
\end{align*}

$\bullet\, K_u=\Q(\alpha_u)$ where $\alpha_u$ is a root of an irreducible polynomial $a_3(u)x^3+a_2(u)x^2 +a_1(u)x+a_0(u)$ for some rational number $u$ where
\begin{align*}
a_3(u)=&27(u-1)^6,\\
a_2(u)=&-27 (u-1)^4 (u^6-6 u^5+19 u^4-40 u^3+63 u^2-66 u+33),\\
a_1(u)=&9(u-1)^2 (u^3-3 u^2+3 u-3) (u^9-9 u^8+44 u^7-146 u^6+354 u^5-648 u^4\\
&+912 u^3-954 u^2+657 u-219),\\
a_0(u)=&-u^{18}+18 u^{17}-165 u^{16}+1020 u^{15}-4716 u^{14}+17172 u^{13}-50904 u^{12}\\
&+125820 u^{11}-263358 u^{10}+470376 u^9-718146 u^8+934740 u^7-1028268 u^6\\
&+939276 u^5-693360 u^4+399924 u^3-173097 u^2+52326 u-8721.
\end{align*}

Next consider the case $E(\Q)=\Z/9\Z$.
This case can be treated by the exact same method as $E(K)_{tors}=\Z/14\Z$ with $E(\Q)_{tors}=\Z/7\Z$.
As in that case, $K$ is not a cyclic cubic field.
One can find the parametrization of $E_u$ with $E_u(\Q)_{tors}=\Z/9\Z$ in \cite[Table 3]{Ku} as follows;

$\bullet\ \  E_u:\,y^2-(u^3-u^2-1)xy-u^2(u-1)(u^2-u+1)y=x^3-u^2(u-1)(u^2-u+1)x^2.$

\noindent with discriminant $\Delta_u= u^9(u-1)^9(u^2-u+1)^3(u^3-6u^2+3u+1)\neq 0$.
Note that $E_u$ is isomorphic to the elliptic curve defined by
\begin{align*}
y^2=f_u(x):=&x^3-\frac{1}{3}(u^3-3u^2+1)(u^9-9u^8+27u^7-48u^6+54u^5-45u^4+27u^3\\
&-9u^2+1)x+\frac{2}{27}(u^{18}-18 u^{17}+135 u^{16}-570 u^{15}+1557 u^{14}\\
&-2970 u^{13}+4128 u^{12}-4230 u^{11}+3240 u^{10}-2032 u^9+1359 u^8\\
&-1080 u^7+735 u^6-306 u^5+27 u^4+42 u^3-18 u^2+1).
\end{align*}
Let $\alpha_u$ be a root of an irreducible polynomial $f_u(x)$ for some rational number $u$.
Then $K_u:=\Q(\alpha_u)$ is a cubic field, so that $E_u(K_u)_{tors}=\Z/18\Z$.
Let $r_1,r_2,r_3$ be the three real roots of $u^3-6u^2+3u+1=0$ with $r_1<r_2<r_3$, then $r_1<0<r_2<1<r_3$.
Note that $u^2-u+1=0$ has no real root.
Put $I:=(-\infty,\,r_1)\cup (0,\,r_2) \cup (1,\, r_3)$, $J:=(r_1,\,0)\cup (r_2,\,1) \cup (r_3,\,\infty)$.
Then $\Delta_u<0$ for $u \in I$ and $\Delta_u>0$ for $u \in J$, hence $E_u$ has torsion $\Z/18\Z$ over 
complex cubic field $K_u$ when $u\in I\cap\Q$ and over totally real, but non-Galois cubic fields $K_u$ when $u\in J\cap\Q$.

Finally, we consider the torsion of $E_u$ over pure cubic fields.
Suppose $u\in I\cap\Q$. Then $K_u$ can be a pure cubic field only if 
\begin{equation}\label{eq:de2}
-27k^2=\Delta_u=u^9(u-1)^9(u^2-u+1)^3(u^3-6u^2+3u+1)
\end{equation}
for some $k\in\Q$ by Lemma \ref{lem:pure}.
By letting $v=\frac{9k}{u^4(u-1)^4(u^2-u+1)}$ in \eqref{eq:de2}, we have the following equation:
$$v^2=-3u(u-1)(u^2-u+1)(u^3-6u^2+3u+1),$$
which defines a hyperelliptic curve $C$ of genus 3.
Since the Chabauty method is implemented in {\sc Magma} only for a curve of genus 2, we use another method to find all the rational points on $C$.
By using {\sc Magma}, we compute that the group ${\rm Aut}_\Q(C)$ of rational automorphisms is of order 6, and it has an automorphism $\sigma$ of order 3 as follows:
$$\sigma(u,v)=\left(\frac{u^4-u^3}{u^4},\,\frac{v}{u^4}\right).$$
Then the quotient curve $C/\langle\sigma\rangle$ is an elliptic curve $E$ defined by
$$y^2=x^3 + 1,$$
and the map $\phi: C\to E$ of degree 3 is given by
$$(x,y)=\phi(u,v)=\left(-\frac{u^3-3u^2+1}{3u(u-1)},\,\frac{v(u^2-u+1)}{9u^2(u-1)^2}\right).$$
Note that $E(\Q)=\{O, (-1,0), (0,\pm 1), (2,\pm 3)\}$ is of order 6.
Moreover, $C$ has three obvious rational points $(0,\,0),\,(1,\,0)$ and $\infty$ which are lying above $O$.
Using the map $\phi$, we can compute explicitly the points on $C$ lying above non-trivial points of $E(\Q)$.
They turn out to be not rational, hence $C(\Q)=\{(0,\,0),\,(1,\,0),\infty\}$.
However, these three points cannot give pure cubic fields.
Therefore, the torsion $\Z/18\Z$ cannot occur over pure cubic fields.

\begin{thm} Suppose $E$ is a rational elliptic curve with $E(K)_{tors}=\Z/18\Z$ over some cubic field $K$.
\begin{itemize}
\item[(a)] If $E(\Q)_{tors}=\Z/6\Z$, then there exist infinitely many non-isomorphic rational elliptic curves $E$ over cyclic cubic fields $K$ so that $E(K)_{tors}=\Z/18\Z$.
\item[(b)] If $E(\Q)_{tors}=\Z/9\Z$, then there exist infinitely many non-isomorphic rational elliptic curves $E$ over both non-Galois totally real cubic and complex cubic fields $K$ so that $E(K)_{tors}=\Z/18\Z$.
But there is no rational elliptic curve $E$ such that $E(K)_{tors}=\Z/18\Z$ over a pure cubic field $K$.
\end{itemize}
\end{thm}

\begin{center}
2.4.\, $E(K)_{tors}=\Z/2\Z\times\Z/14\Z$
\end{center}

Bruin and Najman \cite{BN} proved the following result:
\begin{thm}\label{thm:2-14} \cite[Theorem 1.2]{BN} 
If $E$ is an elliptic curve over a cubic field $K$ with torsion subgroup isomorphic to $\Z/2\Z\times \Z/14\Z$,
then $K$ is cyclic over $\Q$ and $E$ is a base change of an elliptic curve over $\Q$.
\end{thm}
Also, they suggested a method to find a rational model of $E$ from a model over a cyclic cubic field in \cite[Remark 3.10]{BN} as follows;
Given an elliptic curve $E$ over a cubic field $K$ with $E(K)_{tors}=\Z/2\Z\times\Z/14\Z$, if we choose a point $P$ of order $7$ in $E(K)_{tors}$, and write down the unique (long) Weierstrass equation for $E$ such that the points $P$, $2P$ and $4P$ lie on the line $y = 0$ and the points $3P$, $5P$ and $6P$ lie on the line $y =-x$, then this Weierstrass equation has coefficients in $\Q$.
On the other hand, the first author, Kim and Lee \cite{JKL1} provided an infinite family of elliptic curves $E_t$ over cyclic cubic fields $K_t$ with $E_t(K_t)_{tors}=\Z/2\Z\times\Z/14\Z$.
Indeed, there is a typo in the family of \cite{JKL1}, and the first author corrected it in \cite{J}.
By using the method from \cite[Remark 3.10]{BN}, let us find an infinite family of rational elliptic curves $E_u$ over cyclic cubic fields $K_u$ with $E_u(K_u)_{tors}=\Z/2\Z\times\Z/14\Z$.

Firstly, let us write down the family of \cite{JKL1} in a short Weierstrass form, say,
\begin{equation}\label{eq:2-14}
E_t:\ \ y^2=x^3+A(t)x+B(t).
\end{equation}
Here we don't present the coefficients because they are huge and complicated.
The reason we use a short Weierstrass form is that the computer algebra systems could not find a point of order 7 from the given equation in \cite{JKL1}, but we don't know why.

Secondly, using {\sc Magma}, we compute a linear factor of the 7-division polynomial, then using {\sc Maple} we compute a 7-torsion point $P=(x_1,\,y_1)$ as follows:

{\tiny \begin{align*}
x_1=&-\{(t-1)(t+1)^3(t^{11}+5t^{10}+7t^9-53t^8-150t^7+178t^6+1422t^5-906t^4-379t^3-10823t^2+22651t-14001)\alpha_t^2\\
&+(t+1)^2(t^{14}+6t^{13}+7t^{12}-84t^{11}-271t^{10}+330t^9+2879t^8+168t^7-12821t^6-9926t^5+19677t^4+96236t^3\\
&-174941t^2+68918t+26205)\alpha_t-(t^{16}+13t^{15}+63t^{14}+69t^{13}-549t^{12}-1919t^{11}+1227t^{10}+15593t^9+13329t^8\\
&-49369t^7-81699t^6+79599t^5+234489t^4-166773t^3-202663t^2+90019t+134106)\}/\{6144(t-1)^2(t^3+t^2-9t-1)^2\}\\
y_1=&-\{(t-1)(t+1)^3(t^{11}+5t^{10}+7t^9-45t^8-150t^7+114t^6+782t^5+390t^4-2427t^3-1223t^2+1787t+2807)\alpha_t^2\\
&+(t+1)^2(t^{14}+6t^{13}+7t^{12}-76t^{11}-263t^{10}+226t^9+2135t^8+760t^7-7621t^6-9622t^5+19213t^4+18452t^3\\
&-8501t^2-26130t-4971)\alpha_t-(t^{14}+11t^{13}+40t^{12}-14t^{11}-497t^{10}-847t^9+2218t^8+6764t^7-4225t^6\\
&-23171t^5-2172t^4+43778t^3+14481t^2-26521t-26230)(t+1)^2\}/\{256(t-1)(t^3+t^2-9t-1)^3\}
\end{align*}}
where $\alpha_t$ is a root of the following cubic equation:
\begin{equation}\label{eq:defining}
f(x,t)=(t^2-1)x^3+(t^3+2t^2-9t-2)x^2-9(t^2-1)x-t^3-2t^2+9t+2.
\end{equation}
Put $mP=(x_m,\,y_m)$ for $m=1,2,3,4,5,6$.
Let $L_1$(resp. $L_2$) be the line through $P$, $2P$ and $4P$(resp. $3P$, $5P$ and $6P$), and
let $Q=(x_0,\,y_0)$ denote the intersection point of $L_1$ and $L_2$, actually, $y_0=0$.

Thirdly, using {\sc Maple}, we find a change of variables to bring $E_t$ in the form described as above by solving the following system of equations:
\begin{align}\label{eq:change}\nonumber
&p^3y_0+p^2qx_0+s=0,\\ \nonumber
&p^2x_0+r=0,\\ \nonumber
&p^3y_1+p^2qx_1+s=0,\\
&p^3y_2+p^2qx_2+s=0,\\ \nonumber
&-(p^2x_3+r)=p^3y_3+p^2qx_3+s,\\ \nonumber
&-(p^2x_5+r)=p^3y_5+p^2qx_5+s. \nonumber
\end{align}
Here the first two equations mean that $Q$ maps to $(0,\,0)$, the next two equations mean that $P$, $2P$ and $4P$ lie on the line $y=0$, and the last two equations mean that $3P$, $5P$ and $6P$ lie on the line $y=-x$.
Using {\sc Maple}, we obtain the following:
{\tiny \begin{align*}
p=&\{(t-1)(t+1)(t^5+t^4-6t^3-46t^2+53t+29)\alpha_t^2+(t^8+4t^7-4t^6-60t^5-42t^4+492t^3-228t^2-308t-111)\alpha_t\\
&+(t+1)(t^7-2t^6+t^5-58t^4+259t^3-318t^2-5t-6)\}/\{2(t^2+3)(t^6+4t^5+13t^4-40t^3+19t^2+36t+31)\},\\
q=&-1/2,\\
r=&-(t^{12}+4t^{11}+6t^{10}-36t^9-45t^8+168t^7+804t^6-1608t^5+855t^4+788t^3+2166t^2+684t+309)/\\
&\{12(t^2+3)^2(t^8+4t^7+16t^6-28t^5+58t^4-84t^3+88t^2+108t+93)\},\\
s=&(t^{12}+4t^{11}+6t^{10}-36t^9-45t^8+168t^7+804t^6-1608t^5+855t^4+788t^3+2166t^2+684t+309)/\\
&\{24(t^2+3)^2(t^8+4t^7+16t^6-28t^5+58t^4-84t^3+88t^2+108t+93)\}.
\end{align*}}

Lastly, letting $t=u$ and using this change of variables, we obtain an infinite family of rational elliptic curves $E_u$ over cyclic cubic fields $K_u$ with $E_u(K_u)_{tors}=\Z/2\Z\times\Z/14\Z$ as follows:

$\bullet\, E_u:\ \  y^2+xy=x^3+A_2(u)x^2+A_4(u)x+A_6(u)$ where
\begin{align*}
A_2(u)=&\frac{-4(u^6+2u^5+15u^4-20u^3+15u^2+18u+33)(u-1)^2(u+1)^2}{(u^2+3)^3(u^6+4u^5+13u^4-40u^3+19u^2+36u+31)},\\
A_4(u)=&\frac{64(u^6+2u^5+3u^4-20u^3+39u^2+18u+21)(u-1)^6(u+1)^6}{(u^2+3)^6(u^6+4u^5+13u^4-40u^3+19u^2+36u+31)^2},\\
A_6(u)=&\frac{4096(u-1)^{12}(u+1)^{12}}{(u^6+4u^5+13u^4-40u^3+19u^2+36u+31)^3(u^2+3)^9},
\end{align*}

or transformed into short Weierstrass form

$\bullet\, E_u:\ \  y^2=x^3+A(u)x+B(u)$ where
\begin{align*}
A(u)=&-(u^{12}+4u^{11}-10u^{10}-68u^9+3u^8+552u^7+4u^6-2568u^5+2103u^4\\
&+1684u^3+1958u^2+396u+37)/\{48(u^2+3)^3(u^6+4u^5+13u^4-40u^3\\
&+19u^2+36u+31)\},\\
B(u)=&(u^{24}+8 u^{23}+12 u^{22}-120 u^{21}-518 u^{20}+504 u^{19}+5068 u^{18}+568 u^{17}\\
&-24009 u^{16}-15024 u^{15}+62936 u^{14}+183120 u^{13}-550452 u^{12}-851984 u^{11}\\
&+4384056 u^{10}-3808912 u^9+1467519 u^8-4083672 u^7+3590300 u^6\\
&+5512360 u^5+6945498 u^4+2943128 u^3+893052 u^2+120024 u\\
&+3753)/\{864(u^2+3)^6(u^6+4u^5+13u^4-40u^3+19u^2+36u+31)^2\}.
\end{align*}

$\bullet\, K_u=\Q(\alpha_u)$ where $\alpha_u$ is a root of the irreducible polynomial $f(x,u)$ given in \eqref{eq:defining} for some rational number $u$.

By this result together with Theorem \ref{thm:2-14}, we have the following result:

\begin{thm} Suppose $E$ is a rational elliptic curve with $E(K)_{tors}=\Z/2\Z\times\Z/14\Z$ for some cubic field $K$.
Then $K$ is a cyclic cubic field.
Moreover, there exist infinitely many non-isomorphic rational elliptic curves $E$ and cyclic cubic fields $K$ so that $E(K)_{tors}=\Z/2\Z\times\Z/14\Z$.
\end{thm}

As a by-product of all results above, we have the following:

\begin{thm} Any rational elliptic curve does not gain a torsion group not occurring over $\Q$ when the base field is extended to a pure cubic field.
\end{thm}

%showed that the torsion group $\Z/2\Z\times \Z/14\Z$ appears only over cyclic cubic fields, and 

\begin{center}
{\bf Acknowledgment}
\end{center}
The authors are indebted to the referee for finding a mistake/gap in the first version of the proof of Lemma \ref{tor18}.

\
\

%%%%%%%%%%%%%%%%%%%%%%%%
%                      %
%     BIBLIOGRAPHY     %
%                      %
%%%%%%%%%%%%%%%%%%%%%%%%

%\bigskip
%Daeyeol Jeon\\
%Department of Mathematics Education, Kongju National University, Kongju, Chungnam, South Korea\\
%E-mail address: dyjeon@kongju.ac.kr\\

\end{document}